\documentclass[11pt, leqno]{amsart}
\usepackage{amsthm, amsmath}
\usepackage{amssymb} 
\usepackage{amscd}
\usepackage[curve,matrix,arrow,frame,tips]{xy}
\usepackage{verbatim}	
\usepackage{color}
\usepackage{graphicx}	

\newtheorem{thm}{Theorem}[section]
\newtheorem{prop}[thm]{Proposition}
\newtheorem{lemma}[thm]{Lemma}

\newtheorem{rmk}[thm]{Remark}

\newtheorem{cor}[thm]{Corollary}

\newcommand{\cat}[1]{\mathcal #1}

\newcommand{\hs}[1]{\hspace{#1pt}}
\newcommand{\vs}[1]{\vspace{#1pt}}
\newcommand{\tab}{\hs{10}}
\newcommand{\vtab}{\vs{15}}
\newcommand{\id}{\mathbb{I}}

\newcommand{\defeq}{\stackrel{\rm def}{=}}
\renewcommand{\subset}{\subseteq}
\renewcommand{\supset}{\supseteq}
\newcommand{\Z}{\mathbb{Z}}

\newcommand{\iso}{\ \tilde{\to}\ }

\newcommand{\squarediagram}[4]{
\begin{center}
$\begin{CD}
#1    @>>> #2   \\
@VVV       @VVV \\
#3    @>>> #4
\end{CD}$
\end{center}
}
\newcommand{\squarediagramword}[8]{
\begin{center}
$\begin{CD}
#1      @>#5>>  #2     \\
@V#6VV          @VV#7V \\
#3      @>#8>>  #4
\end{CD}$
\end{center}
}

\newcommand{\repn}{representation}
\newcommand{\ga}{\alpha}

\newcommand{\claimend}{\tab $\triangle$}
	
\newcommand{\Q}{\mathbb{Q}}

\newcommand{\C}{\mathbb{C}}
\newcommand{\pt}{{\rm Spec}\, k}
\newcommand{\morp}{morphism}

\newcommand{\wrt}{with respect to}

\newcommand{\withoutlog}{ithout loss of generality}
\newcommand{\longto}{\longrightarrow}

\newcommand{\spec}[1]{{\rm Spec}\, #1}

\newenvironment{statementslist}{\begin{tabular}[t]{p{15pt}p{380pt}}}{\end{tabular}}
\newcommand{\kernel}[1]{{\rm kernel}\, #1}

\newcommand{\homo}{homomorphism}
\renewcommand{\char}[1]{{\rm char\,#1}}

\newcommand{\reduced}[1]{#1_{\rm red}}
\newcommand{\rank}[1]{{\rm rk\,#1}}

\newcommand{\ann}[1]{{\rm Ann\,}#1}
\newcommand{\nil}[1]{{\rm Nil}(#1)}

\newcommand{\equi}{equivariant}

\newcommand{\lazard}{\mathbb{L}}

\newcommand{\rou}{root\ of\ unity}

\newcommand{\fgl}{formal group law}

\begin{document}

\title{On the Equivariant Lazard ring and Tom Dieck's equivariant cobordism ring}
\author{C. L. Liu}
\begin{abstract}
For a torus $G$ of rank $r = 1$, we showed that the canonical ring \homo\ $\lazard_G \to MU_G$, where $\lazard_G$ is the \equi\ Lazard ring and $MU_G$ is Tom Dieck's \equi\ cobordism ring, is surjective. We also showed that the completion map $MU_G \to \widehat{MU}_G \cong MU(BG)$ is injective. Moreover, we showed that the same results hold if we assume a certain algebraic property holds in $\lazard_G$ when $r \geq 2$.
\end{abstract}

\address{Department of Mathematics, University of Duisburg Essen, Thea-Leymann-Strasse 9, Essen, 45127, Germany} 
\email{chun.liu@uni-due.de}
\date{\today}

\maketitle

\medskip

\medskip

\section{Introduction}

Let $G$ be a compact abelian Lie group. In his paper \cite{complex cobor}, Tom Dieck defined an \equi\ cobordism theory $MU^*_{G}(-)$. While it has many nice properties, for instance, there is a Conner-Floyd natural transformation
$$MU^*_{G}(-) \to K^*_{G}(-),$$
it is rather mysterious from the computational point of view. For example, even when $G$ is a cyclic group of order $n > 2$, the ring structure of $MU_{G}$ is not clear, from the algebraic point of view (When $n=2$, see \cite{complex cobor ring when order two} for an explicit algebraic description of $MU_{G}$). It is worth mentioning that although a description of $MU_G$ is given in Sinha's paper \cite{computation of MU}, it depends on an implicit choice of a basis of $MU_G$ as a free $MU$-module.

In their paper \cite{equi FGL}, Cole, Greenlees and Kriz introduced a notion called \equi\ \fgl\ and showed that there is a representing ring, called \equi\ Lazard ring. By Corollary 14.3 in \cite{equi FGL}, this ring $\lazard_{G}(F)$ (relative to a complete $G$-flag $F$) has a very explicit, algebraic description. Moreover, there is a canonical ring \homo
$$\nu_{G} : \lazard_{G}(F) \to MU_{G}.$$
As an analogue of the non-\equi\ case, Greenlees conjectured that this ring \homo\ is an isomorphism (see section 13 in \cite{equi FGL 2}). Since the general case will follow easily from the case when $G$ is a torus, we will focus on this case. 

There are two parts in this paper. Part 1 (section \ref{sec canonical ring homo}) is devoted to the investigation of the injectivity and surjectivity of the canonical ring \homo\ $\nu_G$. In part 2 (section \ref{sec completion map}), we consider the completion map
$$MU_G \to \widehat{MU}_G \cong MU(BG)$$
\wrt\ the ideal generated by the Euler classes. To be more precise, we are interested in showing its injectivity. For both parts, we have satisfactory results when $\rank{G} = 1$ (see Corollary \ref{cor greenlee conj for rank 1} and Theorem \ref{thm completion map}). When $\rank{G} \geq 2$, 
we manage to establish the same results, under the assumption that a certain algebraic property holds for $\lazard_G(F)$ (see Theorem \ref{thm lazard to MU iso}, \ref{thm completion map}).

\bigskip
\bigskip

\section{Notations and assumptions}
\label{sect notation}

In this paper, $G$ is a split torus of rank $r$. Hence, any $G$-\repn\ over $\C$ can be written as direct sum of 1-dimensional $G$-\repn s. We also fix a complete $G$-universe $\cat{U}$ and a complete $G$-flag $F$ given by 
$$0 = V^0 \subset V^1 \subset V^2 \subset \cdots.$$
We denote the 1-dimensional $G$-characters $V^i/V^{i-1}$ by $\ga_i$. All $G$-characters are 1-dimensional unless stated otherwise. For simplicity, we will denote the \equi\ Lazard ring relative to the flag $F$ by $\lazard_G$.

\bigskip
\bigskip

\section{The canonical ring homomorphism $\nu$} 
\label{sec canonical ring homo}

In this section, we will investigate the injectivity and surjectivity of the canonical ring \homo
$$\nu_G : \lazard_G \to MU_G.$$

We will follow the notation in \cite{equi FGL}. For instance, we will denote the structure constants of $\lazard_G(F)$ by $b^{ij}_s, d(\alpha)^i_s, f^i_{st}$. Let $S_G \subset \lazard_G$ be the multiplicative set generated by the non-trivial Euler classes ($e(\beta) = d(\beta)^1_0$ where $\beta$ is non-trivial) and $K_{G,S}$ be the kernel of the localization map $\lazard_G \to S_G^{-1} \lazard_G$. First, we have the following result regarding the injectivity of the map $\nu_G $.

\begin{prop}
\label{prop lazard to MU inj}
The kernel of $\nu_G : \lazard_G \to MU_G$ is $K_{G,S}$.
\end{prop}

\begin{proof}
By Theorem 1.2 in \cite{computation of MU}, the non-trivial Euler classes in $MU_G$ are regular. Therefore, $\nu_G(K_{G,S})=0$. Consider the following commutative diagram :

\squarediagram{\lazard_G/K_{G,S}}{S_G^{-1} \lazard_G}{MU_G}{S_G^{-1}MU_G}

\noindent The top row is certainly injective and the right column is an iso\morp\ by Proposition 13.2 in \cite{equi FGL 2}. The result then follows.
\end{proof}

For the surjectivity of $\nu_G$, we have a satisfactory result when $r=1$.

\begin{prop}
\label{prop surj of map equi lazard to equi MU}
Suppose $G$ is a split torus of rank 1. Then 
$$\nu_G : \lazard_G \to MU_G$$
is surjective.
\end{prop}

\begin{proof}
Consider the following commutative diagram :

\squarediagramword{\lazard_G}{S_G^{-1} \lazard_G}{MU_G}{S_G^{-1} MU_G}{}{\nu_G}{S_G^{-1}\nu_G}{}

\noindent By Proposition 13.2 in \cite{equi FGL 2}, the map $S_G^{-1}\nu_G$ is an iso\morp. Moreover, the bottom row is injective because Euler classes in $MU_G$ are regular (Theorem 1.2 in \cite{computation of MU}). So, for any element $a \in MU_G$, there exist some non-trivial $G$-characters $\gamma_i$ such that $(\prod_i e(\gamma_i)) a$ is in the image of $\nu_G$. Suppose $\nu_G$ is not surjective. Let $M \defeq MU_G / \lazard_G$. Then there exists an element $0 \neq a \in M$ such that $e(\beta^n) a = 0$, where $\beta \in G^*$ is a generator and $n$ is a positive integer. 

Consider the following commutative diagram :

\begin{center}
$\begin{CD}
@. \lazard_G @>e(\beta)>> \lazard_G @>>> \lazard @>>> 0 \\
@. @VVV @VVV @VVV @. \\
0 @>>> MU_G @>e(\beta)>> MU_G @>>> MU @>>> 0 
\end{CD}$
\end{center}

\noindent The top row is exact because $\lazard_G / (e(\beta)) \cong \lazard_{\{1\}} \cong \lazard$. By Theorem 1.2 in \cite{computation of MU}, the bottom row is exact. Furthermore, the right column is an isomorphism by Quillen's Theorem. By the Snake Lemma, the map $e(\beta) : M \to M$ is injective. Therefore, the localization map $M \to S_1^{-1} M$ is injective, where $S_m \subset \lazard_G$ denotes the multiplicative set generated by $e(\beta), \ldots, e(\beta^m)$.

Now, consider a similar commutative diagram with $e(\beta^2)$ instead of $e(\beta)$ and localize it \wrt\ $S_1$ :

\begin{center}
$\begin{CD}
@. S_1^{-1} \lazard_G @>e(\beta^2)>> S_1^{-1} \lazard_G @>>> S_1^{-1} \lazard_{\Z / (2)} @>>> 0 \\
@. @VVV @VVV @VVV @. \\
0 @>>> S_1^{-1} MU_G @>e(\beta^2)>> S_1^{-1} MU_G @>>> S_1^{-1} MU_{\Z / (2)} @>>> 0 
\end{CD}$
\end{center}

\noindent The rows are exact for the same reasons. By Proposition 13.2 in \cite{equi FGL 2}, the right column is an isomorphism. Again, by the Snake Lemma, the map $e(\beta^2) : S_1^{-1} M \to S_1^{-1} M$ is injective. Therefore, $M \to S_1^{-1} M \to S_2^{-1} M$ is injective. Inductively, $M \to S_n^{-1} M$ is injective. But $a \neq 0$ is in its kernel by our assumption. That draws a contradiction and we are done.
\end{proof}

\begin{cor}
\label{cor greenlee conj for rank 1}
Suppose $G$ is a split torus of rank 1 and $H \subset G$ is a closed subgroup.

\noindent \begin{statementslist}
{\rm (1)} & The canonical map $\nu_G : \lazard_G/K_{G,S} \to MU_G$ is an isomorphism. \\
{\rm (2)} & Let $p : \lazard_G \to \lazard_H$ be the canonical map. Then, $\nu_H : \lazard_H \to MU_H$ factors through $\lazard_H / p(K_{G,S})$ and $\nu_H : \lazard_H / p(K_{G,S}) \to MU_H$ is an isomorphism.
\end{statementslist}
\end{cor}

\begin{proof}
Part (1) follows from Proposition \ref{prop lazard to MU inj} and \ref{prop surj of map equi lazard to equi MU}. For part (2), let $\beta \in G^*$ be a generator and $n$ be a generator of the kernel of $\Z \cong G^* \to H^*$. W\withoutlog, $n > 0$. Consider the following commutative diagram :

\begin{center}
$\begin{CD}
@. \lazard_G @>e(\beta^n)>> \lazard_G @>{p}>> \lazard_H @>>> 0 \\
@. @V{\nu_G}VV @V{\nu_G}VV @V{\nu_H}VV @. \\
0 @>>> MU_G @>e(\beta^n)>> MU_G @>>> MU_H @>>> 0 
\end{CD}$
\end{center}

\noindent Since $\nu_G(K_{G,S}) = 0$, we have $\nu_H \circ p(K_{G,S}) = 0$. So $\nu_H$ factors through $\lazard_H / p(K_{G,S})$. Apply $\otimes_{\lazard_G} (\lazard_G / K_{G,S})$ to the top row, we have :

\begin{center}
$\begin{CD}
@. \lazard_G/K_{G,S} @>e(\beta^n)>> \lazard_G/K_{G,S} @>{p}>> \lazard_H/p(K_{G,S}) @>>> 0 \\
@. @V{\nu_G}VV @V{\nu_G}VV @V{\nu_H}VV @. \\
0 @>>> MU_G @>e(\beta^n)>> MU_G @>>> MU_H @>>> 0 
\end{CD}$
\end{center}

\noindent The top row is clearly exact and the bottom row is exact by Theorem 1.2 in \cite{computation of MU}. The result then follows from part (1) and the five Lemma.
\end{proof}

To generalize our results to split torus of higher rank, we need to assume a certain property holds in $\lazard_G$, which is justified by the following generalization of Theorem 1.2 in \cite{computation of MU}. Suppose $H$ is a compact abelian Lie group. We call a set of $H$-characters $\beta_1, \ldots, \beta_s$ linearly independent if their corresponding torus part (as in $\Z^r$, where $r$ is the rank of $H^*$) are $\Z$-linearly independent.

\begin{prop}
\label{prop gen thm by Sinha}
Suppose $H$ is a compact abelian Lie group. Any set of linearly independent $H$-characters $\beta_1, \ldots, \beta_s$ defines a regular sequence $e(\beta_1), \ldots, e(\beta_s)$ in $MU_H$. Moreover, $MU_H / (e(\beta_1), \ldots, e(\beta_s)) \cong MU_{H'}$ where $H' \subset H$ is the closed subgroup corresponding to $H^* / (\beta_1, \ldots, \beta_s)$.
\end{prop}

\begin{proof}
By induction on $s$, see the proof of Theorem 1.2 in \cite{computation of MU} for the case when $s=1$.
\end{proof}

\begin{thm}
\label{thm lazard to MU iso}
Suppose $G$ is a split torus of rank $r \geq 2$. Assume the analogue of Proposition \ref{prop gen thm by Sinha} holds for $\lazard_G$ when $s \leq 2$. Then the canonical map $\nu_G : \lazard_G \to MU_G$ is an isomorphism.
\end{thm}

\begin{proof}
By Proposition \ref{prop lazard to MU inj} and our assumption ($s=1$), it is enough to show the surjectivity. Let $M \defeq MU_G / \lazard_G$. Suppose $\nu_G$ is not surjective. Since $S_G^{-1} M = 0$ (by Proposition 13.2 in \cite{equi FGL 2}), there is an element $0 \neq x \in M$ and a non-trivial $G$-character $\beta$ such that $e(\beta)x=0$ in $M$. By an automorphism of $G$, we may assume $\beta = (n, 0 , \ldots, 0) \in \Z^r \cong G^*$ where $n > 0$. Let $n, x$ be such a pair with minimal $n$.

Let $T \subset \lazard_G$ be the multiplicative set generated by $e(\gamma)$ where $\gamma = (a_1, 0, \ldots, 0)$ and $0 < a_1 < n$. By the minimality of $n$, we have $0 \neq x \in T^{-1}M$ and $e(\beta)x=0$ in $T^{-1}M$. Consider the following commutative diagram of $\lazard_G$-modules :

\medskip

\begin{center}
$\begin{CD}
 @. T^{-1}\lazard_G @>{e(\beta)}>> T^{-1}\lazard_G @>>> T^{-1}\lazard_H @>>> 0 \\
 @. @V{T^{-1}\nu_G}VV @V{T^{-1}\nu_G}VV @V{T^{-1}\nu_H}VV @. \\
0 @>>> T^{-1}MU_G @>{e(\beta)}>> T^{-1}MU_G @>>> T^{-1}MU_H @. 
\end{CD}$
\end{center}

\medskip

\noindent where $H \subset G$ is the closed subgroup corresponding to $G^* / (\beta)$. The rows are exact as before. By the Snake Lemma, the following sequence is exact :
$$\kernel{T^{-1}\nu_H} \to T^{-1}M \stackrel{e(\beta)}{\longto} T^{-1}M.$$
Therefore, it is enough to show $T^{-1}\nu_H$ is injective.

Let $T' \subset \lazard_G$ be the multiplicative set generated by $e(\gamma)$ where $\gamma = (a_1, \ldots, a_r)$ such that at least one of $a_2, \ldots, a_r$ is non-zero. Consider the following commutative diagram :

\squarediagramword{T^{-1}\lazard_H}{T^{-1}MU_H}{T^{-1}{T'}^{-1}\lazard_H}{T^{-1}{T'}^{-1}MU_H}{T^{-1}\nu_H}{}{}{}

\noindent By our assumption ($s=2$), $\lazard_H \to {T'}^{-1} \lazard_H$ is injective, so is the left column. Notice that $T^{-1}{T'}^{-1}\lazard_H = S_H^{-1} \lazard_H$ and the same for $MU_H$. By Proposition 13.2 in \cite{equi FGL 2}, the bottom row is an isomorphism. Hence, $T^{-1}\nu_H$ is injective and we are done.
\end{proof}

\begin{cor}
\label{cor greenlee conj for higher rank}
Suppose $G$ is a split torus of rank $r \geq 2$. Assume the analogue of Proposition \ref{prop gen thm by Sinha} holds for $\lazard_G$ when $s \leq 2$. Then, for any closed subgroup $H \subset G$, the canonical map $\nu_H : \lazard_H \to MU_H$ is an isomorphism.
\end{cor}

\begin{proof}
Suppose $H \subset G$ is the closed subgroup corresponding to $G^* / (\beta)$ for some non-trivial $G$-character $\beta = (n, 0, \ldots, 0) \in \Z^r \cong G^*$. Consider the following commutative diagram :

\medskip

\begin{center}
$\begin{CD}
\lazard_G @>{e(\beta)}>> \lazard_G @>>> \lazard_H @>>> 0 \\
@V{\nu_G}VV @V{\nu_G}VV @V{\nu_H}VV @. \\
MU_G @>{e(\beta)}>> MU_G @>>> MU_H @>>> 0 
\end{CD}$
\end{center}

\medskip

\noindent The top row is exact as before and the bottom row is exact by Proposition \ref{prop gen thm by Sinha}. By Theorem \ref{thm lazard to MU iso}, the left and middle columns are isomorphisms, then so is the right column. The result for arbitrary $H$ can be shown by repeated applications of the above argument.
\end{proof}

\begin{rmk}
{\rm
By Corollary \ref{cor greenlee conj for rank 1} and \ref{cor greenlee conj for higher rank}, if $G$ is a split torus of rank $r$, $H \subset G$ is a closed subgroup and we assume the analogue of Proposition \ref{prop gen thm by Sinha} holds for $\lazard_G$ when $s \leq 2$ in the case when $r \geq 2$, then the canonical map $\nu_H : \lazard_H \to MU_H$ is surjective. By Theorem 16.1 in \cite{equi FGL}, $\lazard_H$ is generated by the Euler classes $e(\beta)$ and the structure constant $f^1_{ij}$. Therefore, $MU_H$, as a ring, is generated by the Euler classes $e(\beta)$ and $\nu_H(f^1_{ij})$. For a more explicit description of $\nu_H(f^1_{ij})$, see remark 11.4 in \cite{universal equi alg cobor}.
}
\end{rmk}

\bigskip

\section{The completion map}
\label{sec completion map}

Our main goal in this section is to show that the completion map 
$$MU_G \to \widehat{MU}_G \cong MU(BG),$$
with respect to the ideal generated by the Euler classes, is injective. Since we will consider some algebraic cobordism theories, we need to fix a ground field $k$. In this section, all schemes (or $G$-schemes) are over $k$. We will also assume $\char{k} = 0$.

Since there is a ring \homo\ $MU \to MU_G$, one would expect to also have a ring \homo\ $\lazard \to \lazard_G$. But it is not clear how to put a non-\equi\ \fgl\ over $\lazard_G$. Therefore, we will consider the algebraic cobordism rings instead. 

Let $\Omega(-)$ be the algebraic cobordism theory defined in \cite{universal alg cobor}, $\omega(-)$ be the algebraic cobordism theory defined by double point relation, as in \cite{alg cobor by DPR}, $\Omega^G(-)$ be the \equi\ algebraic cobordism theory defined in \cite{universal equi alg cobor} and $\Omega$, $\omega$, $\Omega^G$ be the corresponding cobordism rings over $\spec{k}$. By Theorem 1 in \cite{alg cobor by DPR}, $\omega(-) \cong \Omega(-)$. Also, the canonical ring \homo\ $\lazard \to \Omega$ is an isomorphism, by Theorem 1.2.7 in \cite{universal alg cobor}. 

Since double point relation holds in $\Omega^G(-)$ (Proposition 5.5 in \cite{universal equi alg cobor}), we have a canonical ring \homo 
$$\omega \to \Omega^G$$
(given by putting trivial $G$-action) as desired.

Let $\Omega^G_{\rm Tot}(-)$ be the \equi\ algebraic cobordism theory defined in \cite{homo equi alg cobor}. By Proposition 8.5 in \cite{universal equi alg cobor}, there is a ring \homo\ 
$$\Omega^G \to \Omega^G_{\rm Tot}.$$ 
Moreover, by subsection 3.3.1 in \cite{homo equi alg cobor}, 
$$\Omega^G_{\rm Tot} \cong \Omega[[e_1, \ldots, e_r]],$$
where $\beta_1, \ldots, \beta_r$ is a set of generator of $G^*$ and $e_i$ corresponds to the Euler class $e(\beta_i) = c(\beta_i)[\id_{\pt}]$.

\begin{lemma}
\label{lemma composition preserve Lazard}
The composition
$$\lazard \cong \omega \to \Omega^G \to \Omega^G_{\rm Tot} \cong \omega[[e_1, \ldots, e_r]] \cong \lazard[[e_1, \ldots, e_r]]$$
preserves elements in $\lazard$.
\end{lemma}

\begin{proof}
The result follows from the fact that the composition
$$\omega \to \Omega^G \to \Omega^G_{\rm Tot} \cong \omega[[e_1, \ldots, e_r]]$$
preserves elements in $\omega$.
\end{proof}

Now, let $I \subset \lazard_G$ be the ideal generated by the Euler classes and $\hat{\lazard}_G$, $\hat{\Omega}^G$ be the $I$-adic completions of $\lazard_G$, $\Omega^G$ respectively. Since $\Omega^G_{\rm Tot}$ is $I$-adically complete, we have an induced ring \homo 
$$\hat{\Omega}^G \to \Omega^G_{\rm Tot}.$$

\begin{prop}
\label{prop completion map for equi alg cobor is inj}
Suppose $G$ is a split torus and $\char{k} = 0$.

\noindent \begin{statementslist}
{\rm (1)} & The ring \homo\  $\hat{\Omega}^G \to \Omega^G_{\rm Tot}$ is an isomorphism. \nonumber\\
{\rm (2)} & The completion map $\Omega^G \to \hat{\Omega}^G$ is injective. \nonumber
\end{statementslist}
\end{prop}

\begin{proof}
For part (1), the map $\lazard \cong \omega \to \Omega^G$ induces a map $f : \lazard[e_1, \ldots, e_r] \to \Omega^G$, by sending $e_i$ to $e(\beta_i)$. So, we have
$$\lazard[e_1, \ldots, e_r] \stackrel{f}{\to} \Omega^G \to \hat{\Omega}^G \to \Omega^G_{\rm Tot} \cong \lazard[[e_1, \ldots, e_r]],$$
which is just the completion map \wrt\ the ideal $(e_1, \ldots, e_r)$ (by Lemma \ref{lemma composition preserve Lazard}).

By Theorem 6.12 in \cite{universal equi alg cobor}, the map $\lazard_G \to \Omega^G$ is surjective, so is $\hat{\lazard}_G \to \hat{\Omega}^G$. By Theorem 6.5 in \cite{equi FGL 2}, $\hat{\lazard}_G \cong \lazard[[e_1, \ldots, e_r]]$. Therefore, we have
$$\hat{\lazard}_G \to \hat{\Omega}^G \to \Omega^G_{\rm Tot} \cong \lazard[[e_1, \ldots, e_r]] \cong \hat{\lazard}_G,$$
which is the identity map (can be seen by considering its quotient by $I^n$). Hence, $\hat{\lazard}_G \to \hat{\Omega}^G$ is an isomorphism and so is $\hat{\Omega}^G \to \Omega^G_{\rm Tot}$.

For part (2), consider the composition
$$\lazard[e_1, \ldots, e_r] \stackrel{f}{\to} \Omega^G \to \hat{\Omega}^G \stackrel{g}{\cong} \lazard[[e_1, \ldots, e_r]]$$
($g$ is an isomorphism by part (1)). Since 
$$\lazard[e_1, \ldots, e_r] / (e_1, \ldots, e_r)^n \to \Omega^G / I^n \to \lazard[[e_1, \ldots, e_r]] / (e_1, \ldots, e_r)^n$$
is the identity map, the map $\lazard[e_1, \ldots, e_r] / (e_1, \ldots, e_r)^n \to \Omega^G / I^n$ is injective for all $n \geq 1$. Therefore, $f^{-1}(I^n) = (e_1, \ldots, e_r)^n$. If $a$ is an element in $\cap_{n \geq 1} I^n \subset \Omega^G$, then 
$$f^{-1}(a) \subset f^{-1}(\cap_n I^n) = \cap_n f^{-1}(I^n) = \cap_n (e_1, \ldots, e_r)^n = 0.$$ 
Hence, $a = 0$ and we are done.
\end{proof}

\begin{thm}
\label{thm completion map}
Suppose $G$ is a spit torus of rank $r$ and $k = \C$.

\noindent
\begin{statementslist}
{\rm (1)} & The canonical map $\Omega^G \to MU_G$ is injective. \nonumber\\
{\rm (2)} & The kernel of the canonical, surjective map $\lazard \to \Omega^G$ is $K_{G,S}$. \nonumber
\end{statementslist}

\medskip

If $r \geq 2$ and we assume the analogue of Proposition \ref{prop gen thm by Sinha} holds for $\lazard_G$ when $s \leq 2$, then

\noindent
\begin{statementslist}
{\rm (3)} & $\lazard_G \iso \Omega^G \iso MU_G$ \nonumber\\
{\rm (4)} & The completion maps $\lazard_G \to \hat{\lazard}_G$ and $MU_G \to \widehat{MU}_G$ are both injective. \nonumber\\
{\rm (5)} & $MU_G$ is an integral domain. \nonumber
\end{statementslist}

\medskip

If $r = 1$, then

\noindent
\begin{statementslist}
{\rm (3')} & $\Omega^G \iso MU_G$ \nonumber\\
{\rm (4')} & The completion map $MU_G \to \widehat{MU}_G$ is injective. \nonumber\\
{\rm (5')} & $MU_G$ is an integral domain. \nonumber
\end{statementslist}
\end{thm}

\begin{proof}
For part (1), consider the following commutative diagram :
\squarediagram{\Omega^G}{\hat{\Omega}^G}{MU_G}{\widehat{MU}_G}
The right column is the composition $\hat{\Omega}^G \to \Omega^G_{\rm Tot} \cong \hat{\lazard}_G \to \widehat{MU}_G$, which is an isomorphism (by part (1) of Proposition \ref{prop completion map for equi alg cobor is inj} and Proposition 13.3 in \cite{equi FGL 2}). The result then follows from part (2) of Proposition \ref{prop completion map for equi alg cobor is inj}. 

For part (2), notice that we have
$$\lazard_G \stackrel{a}{\to} \Omega^G \stackrel{b}{\to} MU_G$$
such that $a$ is surjective (Theorem 6.12 in \cite{universal equi alg cobor}) and $b$ is injective (by part (1)). The result then follows from Proposition \ref{prop lazard to MU inj}. 

Part (3) follows from Corollary \ref{cor greenlee conj for higher rank} and part (2). Part (4) follows from part (2) of Proposition \ref{prop completion map for equi alg cobor is inj} and part (3). Part (5) follows from part (4) and the fact that 
$$\widehat{MU}_G \cong \lazard[[e_1, \ldots, e_r]]$$ 
is an integral domain. Part (3'), (4'), (5') follow from similar arguments as those for part (3), (4), (5) respectively.
\end{proof}

\begin{rmk}
{\rm 
Many results are only true for split torus. Even for cyclic group of order $n$, $\lazard_G$ can behave quite badly (so is $MU_G$). For example, if $\beta \in G^*$ is a generator, $0 \neq e(\beta) \in \lazard_G$ is not regular (because $0 = e(\beta^n) = e(\beta)(n + a e(\beta))$ for some $a$). Moreover, if $n = 6$, one can show that $S_G^{-1} \hat{\lazard}_G$ is the zero ring. In particular, that means $\phi$ is not injective when $n = 6$, because $S_G^{-1} \lazard_G$ is never the zero ring by Corollary 6.4 in \cite{equi FGL 2}. 
}
\end{rmk}

\bigskip

\bigskip


\end{document}